\documentclass[leqno,11pt]{scrartcl}
\usepackage[utf8]{inputenc}

\usepackage{amsmath,amssymb,amsthm}
\usepackage{amscd}
\usepackage{setspace} 
\usepackage{enumerate}
\usepackage{array}
\usepackage{tabularx}
\usepackage{multirow}
\usepackage{booktabs}
\usepackage{url}

\theoremstyle{definition}
\newtheorem{definition}{Definition}
\newtheorem{remark}{Remark}

\theoremstyle{plain}
\newtheorem{theorem}{Theorem}

\newtheorem{lemma}[definition]{Lemma}
\newtheorem{corollary}{Corollary}

\theoremstyle{remark}

\newcommand{\C}{\mathbb{C}}

\renewcommand{\H}{\mathbb{H}}
\newcommand{\K}{\mathbb{K}}
\newcommand{\N}{\mathbb{N}}

\newcommand{\Q}{\mathbb{Q}}
\newcommand{\R}{\mathbb{R}}
\newcommand{\Z}{\mathbb{Z}}

\newcommand{\Ccal}{\mathcal{C}}

\newcommand{\Hcal}{\mathcal{H}}
\newcommand{\Ical}{\mathcal{I}}
\newcommand{\Mcal}{\mathcal{M}}

\newcommand{\Tcal}{\mathcal{T}}

\newcommand{\trace}{\operatorname{trace}\,}
\newcommand{\diag}{\operatorname{diag}\,}

\newcommand{\oh}{\mathcal{\scriptstyle{O}}}

\let\leq\leqslant
\let\geq\geqslant

\begin{document}

\begin{center}
\begin{huge}
\begin{spacing}{1.0}
\textbf{On Hecke theory for Hermitian modular forms}  
\end{spacing}
\end{huge}

\bigskip
by
\bigskip

\begin{large}
\textbf{Adrian Hauffe-Waschbüsch} and \textbf{Aloys Krieg}\footnote{Adrian Hauffe-Waschbüsch, Aloys Krieg\\ Lehrstuhl A für Mathematik, RWTH Aachen University, D-52056 Aachen, Germany \\ adrian.hauffe@matha.rwth-aachen.de, krieg@rwth-aachen.de}\\
\bigskip\bigskip\bigskip
\emph{Dedicated to Murugesan Manickam on the occasion of his 60th birthday.}
\end{large}
\vspace{2cm}
\end{center}
\noindent\textbf{Abstract.}
In this paper we outline the Hecke theory for Hermitian modular forms in the sense of Hel Braun for arbitrary class number of the attached imaginary-quadratic number field. The Hecke algebra turns out to be commutative. Its inert part has a structure analogous to the case of the Siegel modular group and coincides with the tensor product of its $p$-components for inert primes $p$. This leads to a characterization of the associated Siegel-Eisenstein series. The proof also involves Hecke theory for particular congruence subgroups.
\medskip

\vspace{2ex}

\newpage

\section{Introduction}

The Hermitian modular group associated with an imaginary-quadratic number field $\K$ was introduced by H. Braun \cite{B}, \cite{B2} as an analogue of the Siegel modular group. The case of class number $>1$ leads to number theoretical complications. If one wants to consider the Hecke theory as for instance by Freitag \cite{F}, there are only a few concrete results (cf. \cite{En}, \cite{K6}). Most authors consider the situation over local fields (cf. \cite{Rm}). 

In this paper we show that each double coset contains a matrix in block diagonal form. Hence the Hecke algebra is commutative. Moreover we characterize a particular subalgebra of the Hecke algebra, which is related to inert primes. As a consequence we obtain a characterization of the Siegel-Eisenstein series, which was available up to now only in the case of class number $1$ (cf. \cite{K4}). Many of our results are similar to the investigations by M. Manickam \cite{Ma1} on Jacobi forms.

\section{The Hecke algebra for the Hermitian modular group}

Throughout the paper let 
\[
 \K = \Q(\sqrt{-m}) \subset \C,\;\, m\in\N\;\;\text{squarefree},
\]
be an imaginary-quadratic number field. Its discriminant and ring of integers are
\[
d_\K = \begin{cases}
           -m \\
           -4m
         \end{cases}\;\;
\text{and}\;\; \oh_\K = \Z+\Z\omega_\K = \begin{cases}
                                         \Z+\Z(1+\sqrt{-m})/2 & \text{if}\;\,m\equiv 3\,(\bmod 4), \\
                                         \Z+\Z\sqrt{-m} & \text{if} \;\, m\equiv 1,2 \,(\bmod 4).
                                        \end{cases}
\]
Denote its class number by $h_\K$ and the associated primitive real Dirichlet character $\mod |d_\K|$ by $\chi_\K$.

Define the set of integral unitary similitudes of factor $q\in\N$ by
\[
 \Delta_n(q):= \{M\in\oh_\K^{2n\times 2n};\; J[M]:= \overline{M}^{tr} JM=qJ\},\;\;J=\begin{pmatrix}
                                                                                     0 & -I \\ I & 0
                                                                                    \end{pmatrix}, \;
I=\begin{pmatrix}
   1 & & 0 \\  & \ddots & \\ 0 & & 1
  \end{pmatrix}.
\]
Moreover let 
\[
 \Delta_{n,q}:=\bigcup_{\ell=0}^{\infty} \Delta_n(q^\ell), \;\;\Delta_n=\bigcup_{q\in\N} \Delta_n(q).
\]
\[
\Gamma_n:= \Delta_n(1) \subseteq U(n,n;\C):=\{M\in\C^{2n\times 2n};\; J[M]=J\}
\]
denotes the \emph{Hermitian modular group of degree $n$}. Given $q\in \N$ let 
\[
\Gamma_n[q]=\{M\in \Gamma_n;\;M\equiv I (\bmod q)\}
\]
stand for the \emph{principal congruence subgroup} of level $q$. We will always assume a block decomposition
\[
 M= \begin{pmatrix}
     A & B \\ C & D
    \end{pmatrix} \in \Delta_n,\;\, A,B,C,D\in \oh_\K^{n\times n}.
\]

\begin{lemma}\label{Lemma_1}
Given $M\in\Delta_n(q)$ then 
\[
 \sharp(\Gamma_n\backslash \Gamma_n M\Gamma_n) < \infty.
\]
\end{lemma}

\begin{proof}
 Use $M^{-1} \Gamma_n M\cap \Gamma_n\supseteq \Gamma_n[q]$, hence
 \[
  \sharp(\Gamma_n\backslash \Gamma_n M \Gamma_n) = \bigl[\Gamma_n: \Gamma_n\cap M^{-1} \Gamma_n M]\leq [\Gamma_n:\Gamma_n[q]\bigr]\leq q^{8n^2} < \infty .
 \]
 \vspace*{-7ex}\\
\end{proof}
Hence $(\Gamma_n,\Delta_n)$ fulfills the Hecke-condition (cf. \cite{F}, \cite{K2}).

Let $\partial_k(G)\subseteq \oh_\K$ stand for the ideal generated by all $k\times k$ subdeterminants of an integral matrix $G$, which is invariant under multiplikation with unimodular matrices. Then \cite{B}, Theorem 1, resp. \cite{B2}, Lemma 1, implies

\begin{lemma}\label{Lemma_2}
If $M\in\Delta_n$ there exist $L^*,L'\in\Gamma_n$ such that 
\begin{align*}
  L^*M = \begin{pmatrix}
        A^* & B^* \\ 0 & D^*
       \end{pmatrix},\;\;
\oh_\K \det A^* = \partial_n\begin{pmatrix}
                           A \\ C
                          \end{pmatrix}, \\
  ML' = \begin{pmatrix}
        A' & 0 \\ C' & D'
       \end{pmatrix},\;\;
\oh_\K \det A' = \partial_n (A,B).                        
\end{align*}
\end{lemma}

The next step is a block diagonal decomposition in double cosets.

\begin{lemma}\label{Lemma_3}
Given $M\in\Delta_n$ there exist $L_1,L_2\in\Gamma_n$ such that 
\[
L_1 M L_2 = \begin{pmatrix}
             A^* & 0 \\ 0 & A^* H
            \end{pmatrix}\;\;
\text{for some} \;\; H= \overline{H}^{tr} \in\oh_\K^{n\times n}.
\]
\end{lemma}

\begin{proof}
Choose $A^*$ such that $|\det A^*|$ is minimal among all the matrices
\[
 \begin{pmatrix}
  A & B \\ C & D
 \end{pmatrix} \in \Gamma_n M \Gamma_n \;\;\text{with}\;\; \det A \neq 0.
\]
Let $M^*= \left(\begin{smallmatrix}
                 A^* & B^* \\ C^* & D^*
                \end{smallmatrix}\right) \in \Gamma_n M \Gamma_n$. Then
\[
M^*\Gamma_n = \begin{pmatrix}
               A' & 0 \\ C' & D'
              \end{pmatrix}\Gamma_n \;\;\text{and}\;\; \partial_n (A^*,B^*) = (\det A') \oh_\K
\]
follow from Lemma \ref{Lemma_2}. However the minimality of $|\det A^*|$ shows $(\det A^*)\oh_{\K} = (\det A')\oh_{\K}$. The same holds for the first block column. Hence
\[
 A^{*-1} B^* \;\;\text{and}\;\; C^* A^{*-1}
\]
are integral and Hermitian. Therefore we get a matrix
\[
 \begin{pmatrix}
  A^* & 0 \\ 0 & D'
 \end{pmatrix} \in\Gamma_n M \Gamma_n.
\]
As $\left(\begin{smallmatrix}
                 A^* & D' \\ 0 & D'
                \end{smallmatrix}\right) \in \Gamma_n M \Gamma_n$ we conclude  
\[
 \partial_n(A^*,D') = (\det A^*)\oh_{\K}
\]
hence    
\[
A^{*-1} D' = H \in \oh_\K^{n\times n} \;\;\text{and} \;\;H=\overline{H}^{tr}.
\]
\vspace*{-7ex}\\
\end{proof}

A simple consequence is

\begin{corollary}\label{Corollary_1} 
Given $M\in\Delta_n$ then
\[
 \Gamma_n M \Gamma_n = \Gamma_n M^{tr} \Gamma_n.
\]
\end{corollary}

\begin{proof}
We assume $\left(\begin{smallmatrix}
                 A & 0 \\ 0 & D
                \end{smallmatrix}\right) \in \Gamma_n M \Gamma_n$ due to Lemma \ref{Lemma_3}. By means of  \cite{En2}, Theorem 2.2, there are $U,V\in GL_n(\oh_\K)$ such that 
\[
UAV = A^{tr}.
\]
Hence
\[
\begin{pmatrix}
 U & 0 \\ 0 & \overline{U}^{tr-1}
\end{pmatrix}
\begin{pmatrix}
 A & 0 \\ 0 & D
\end{pmatrix}
\begin{pmatrix}
 V & 0 \\ 0 & \overline{V}^{tr-1}
\end{pmatrix} = \begin{pmatrix}
		  A^{tr} & 0 \\ 0 & D^*
		 \end{pmatrix},
\]
$J[M] = qJ$ then implies $\overline{A}^{tr}D = \overline{A}D^* = qI$, hence $D^* = D^{tr}$.
\end{proof}

As $M\mapsto M^{tr}$ is an involution which keeps the double cosets invariant, we conclude from \cite{F} or \cite{K2} the following

\begin{theorem}\label{Theorem_1} 
$(\Gamma_n,\Delta_n)$ is a Hecke pair. The Hecke algebra $\Hcal(\Gamma_n,\Delta_n)$ is commutative.
\end{theorem}

Our next aim is to describe particular products in this Hecke algebra. Therefore we need

\begin{lemma}\label{Lemma_4} 
Let $q,r \in\N$ be coprime and $d_\K \neq -3,-4$. Then 
\[
\Gamma_n[q]\cdot \Gamma_n[r] = \Gamma_n.
\]
\end{lemma}

\begin{proof}
As the principal congruence subgroups are normal, we may restrict to generators of $\Gamma_n$. We use the generators from \cite{De}, Theorem 2.1, for which the claim follows by a simple calculation of the form
\[
 \begin{pmatrix}
  I & \ell H \\ 0 & I
 \end{pmatrix}
  \begin{pmatrix}
  I & H \\ 0 & I
 \end{pmatrix} \in\Gamma_n[r]
\]
for $H =\overline{H}^{tr} \in \oh_\K^{n\times n}$ and some $\ell\in \N$, $q\mid\ell$.
\end{proof}

An application is described in 

\begin{corollary}\label{Corollary_2} 
Given $M\in \Delta_n(q)$, $r\in \N$, $\gcd(q,r) = 1$ then
\[
 \Gamma_n M \Gamma_n = \Gamma_n M \Gamma_n[r].
\]
\end{corollary}

\begin{proof}
Clearly $M^{-1} \Gamma_n M \cap \Gamma_n \supseteq \Gamma_n[q]$ holds. Now apply Lemma \ref{Lemma_4}.
\end{proof}

We consider a particular case. Let $M\in\Delta_n(q)$, $\gcd(q,r)=1$ and $M\equiv I (\bmod\,r)$ as well as
\[
\Gamma_n M \Gamma_n = \overset{\bullet}{\bigcup\limits_{1\leq j\leq \ell}} \Gamma_n L_j,\;\; L_j\equiv I (\bmod \,r)
\]
due to Corollary \ref{Corollary_2}. Then we immediately obtain
\begin{gather*}\tag{1}\label{gl_1}
\Gamma_n[r] M \Gamma_n[r] = \overset{\bullet}{\bigcup\limits_{1\leq j\leq \ell}} \Gamma_n[r] L_j
\end{gather*}
as well as
\begin{gather*}\tag{2}\label{gl_2}
 \Gamma_n M \Gamma_n = \overset{\bullet}{\bigcup\limits_{1\leq j\leq \ell}} \Gamma_n RL_j R^{-1}\;\; \text{for}\;\; R\in \Delta_n(r).
\end{gather*}

An immediate consequence is 

\begin{corollary}\label{Corollary_3} 
Given $M\in\Delta_n(q)$, $L\in\Delta_n(r)$ with coprime $q,r\in \N$, then
\[
 \Gamma_n M \Gamma_n \cdot \Gamma_n L \Gamma_n = \Gamma_n M L \Gamma_n.
\]
\end{corollary}

\begin{proof}
We choose decompositions
\[
 \Gamma_n M \Gamma_n = \bigcup_i \Gamma_n M K_i, \; K_i \in \Gamma_n[r],\;\; \Gamma_n L \Gamma_n = \bigcup_j \Gamma_n L R_j
\]
due to Corollary \ref{Corollary_2}. Clearly the right cosets
\[
 \Gamma_n M K_i L R_j
\]
are mutually disjoint and contained in $\Gamma_n M L \Gamma_n$. Thus the claim follows.
\end{proof}

In the case of $h_\K = 1$ the Hecke algebra coincides with the tensor product of its primary components
\[
\Hcal_{n,p} = \Hcal (\Gamma_n,\Delta_{n,p}), \;\; p\;\text{prime}.
\]
In this situation the structure is described in \cite{K6}. If $h_\K> 1$ this result is no longer true (cf. \cite{En}, 3.3.6), e.g. for $\K=\Q(\sqrt{-5})$
\[
 \Gamma_2 \diag(1,1+\sqrt{-5},6,1+\sqrt{-5}) \Gamma_2 \not\in \bigotimes\limits_p \Hcal_{n,p}.
\]
Many authors define the Hecke algebra as the tensor product of its $p$-components (cf. \cite{Rm}). But the tensor product is a proper subalgebra of $\Hcal(\Gamma_n,\Delta_n)$ in general.

The example shows that it is much more difficult to look at the decomposition of double cosets.

\begin{lemma}\label{Lemma_5} 
Let $M\in \Delta_n(q)$, $q=r_1 r_2\in \N$, where $r_1$ is a product of split or ramified primes and $r_2$ a product of inert primes. Then there exist $M_j\in \Delta_n(r_j)$, $j=1,2$, such that
\[
\Gamma_n M \Gamma_n = \Gamma_n M_1 \Gamma_n \cdot \Gamma_n M_2 \Gamma_n.
\]
\end{lemma}

\begin{proof}
We may assume $M=\left(\begin{smallmatrix}
                        A & 0 \\ 0 & D
                       \end{smallmatrix}\right)$ 
due to Lemma \ref{Lemma_3} and consider the determinantal divisors. Let
\[
\partial_k(A) = \Ical_k\cdot \oh_\K a_k,\;\,k=1,\ldots,n,
\]
where $a_k\in\N$ divides $r_2^n$ and $\Ical_k$ is not divisible by $p\oh_\K$ for any inert prime $p$. In view of \cite{En2}, Theorem 2.1, there exist
\[
 A_j\in\oh_\K^{n\times n},\; j=1,2,\;\; \partial_k(A_1) = \Ical_k,\;\;\partial_k(A_2) = \oh_\K a_k,\; k=1,\ldots,n.
\]
Define $D_j= r_j\overline{A}_j^{tr-1}$. Then we have
\begin{align*}
 & M_j = \begin{pmatrix}
          A_j & 0 \\ 0 & D_j
         \end{pmatrix} \in\Delta_n(r_j),\; j=1,2 \,, \\
         & \Gamma_n M_1 \Gamma_n \cdot \Gamma_n M_2 \Gamma_n = \Gamma_n M_1 M_2 \Gamma_2
\end{align*}
by means of Corollary \ref{Corollary_3}. As $\Ical_k$ and $\oh_\K a_k$ are coprime, we conclude
\[
\partial_k(A_1A_2) = \partial_k(A_1)\cdot \partial_k(A_2) = \partial_k(A), \; k=1,\ldots,n\,,
\]
from \cite{En2}, Theorem 4.2, or \cite{En}, Satz 2.6.8. Then
\[
\Gamma_n M_1 M_2 \Gamma_2 = \Gamma_n M \Gamma_n
\]
follows from \cite{En2}, Theorem 2.2.
\end{proof}

\section{The inert part of the Hecke algebra}

Lemma \ref{Lemma_5} shows that it is interesting to have a closer look at the inert part defined by
\[
 \Delta_n^{\text{inert}} = \bigcup_{\substack{q \in \N \\ p\mid q \Rightarrow p\;\text{inert}} }
 \Delta_n(q)
\]
and call
\[
 \Hcal_n^{\text{inert}} = \Hcal(\Gamma_n, \Delta_n^{\text{inert}})
\]
the inert part of the Hecke algebra.

Given $M\in \Delta_n(q)$, where $q$ is only divided by inert primes, we conclude that $\partial_k(M) = \oh_\K r$, where $r\mid q^n$. Thus we can apply Theorem \ref{Theorem_1} as well as \cite{En2}, Theorem 2.2, in order to obtain the elementary divisor theorem similar to the case of the Siegel modular group (cf. \cite{F}, \cite{K2}).

\begin{theorem}\label{Theorem_2} 
Given $M\in\Delta_n(q) \subseteq \Delta_n^{\text{inert}}$ the double coset $\Gamma_n M \Gamma_n$ contains a unique representative
\begin{align*}
 & \diag(a_1,\ldots,a_n,d_1,\ldots,d_n),\;a_j,d_j \in\N,\;a_j d_j=q \,, \\
 & a_1 |a_2|\ldots, |a_n|d_n|d_{n-1}| \ldots | d_1.
\end{align*}
\end{theorem}

In this case the elementary divisor theorem holds. Next we have a look at right coset representatives.

\begin{corollary}\label{Corollary_5} 
Given $M\in\Delta_n(q) \subseteq \Delta_n^{\text{inert}}$ the right coset $\Gamma_n M$ possesses a representative of the form
\[
 \begin {pmatrix}
  A & B \\ 0 & D
 \end {pmatrix}, \;\; A= \overline{D}^{tr-1},
\]
where $D$ is an upper triangular matrix with diagonal entries $d_j\in\N$, $d_j\mid q$, $j=1,\ldots,n$.
\end{corollary}

Now we use Corollary \ref{Corollary_3} in order to get

\begin{corollary}\label{Corollary_4} 
$\Hcal_n^{\text{inert}} = \bigotimes\limits_{p\;\text{inert}} \Hcal (\Gamma_n,\Delta_{n,p})$.
\end{corollary}

In this case one can directly adopt the proofs, which are given for the Siegel modular group in \cite{F} or \cite{K2}.
\bigskip

Next we consider generators.

\begin{corollary}\label{Corollary_6} 
Let $p$ be an inert prime. Then $\Hcal (\Gamma_n,\Delta_{n,p})$ is generated by the double cosets
\begin{align*}
 \Tcal_n(p) & = \Gamma_n\begin{pmatrix}
                         I & 0 \\ 0 & pI
                        \end{pmatrix} \Gamma_n, \\
 \Tcal_{n,j}(p^2) & = \Gamma_n \diag\bigl(\underset{j}{\underbrace{1,\ldots,1}},\underset{n-j}{\underbrace{p,\ldots,p}},
 \underset{j}{\underbrace{p^2,\ldots,p^2}},\underset{n-j}{\underbrace{p,\ldots,p}}\bigr) \Gamma_n\;\,j=0,\ldots,n-1,
\end{align*}
which are algebraically independent.
\end{corollary}

Given $M\in \Delta_n(q) \subseteq \Delta_n^{\text{inert}}$ we choose a representative
\[
 M^* = \begin{pmatrix}
        A & B \\ 0 & D
       \end{pmatrix}, \;\;
 A = \begin{pmatrix}
        A_1 & 0 \\ \overline{a}^{tr} & \alpha
       \end{pmatrix}, \;\;
 B= \begin{pmatrix}
        B_1 & \ast \\ \ast & \ast
       \end{pmatrix}, \;\;
 D = \begin{pmatrix}
        D_1 & d \\ 0 & \delta
       \end{pmatrix}    
\]
in $\Gamma_nM$ and define for $k\in\Z$, $n\geq 2$
\[
 \phi_k(\Gamma_n M) = \delta^{-k} \Gamma_{n-1} M_1 ,\;\; M_1 = \begin{pmatrix}
                                                                A_1 & B_1 \\ 0 & D_1
                                                               \end{pmatrix}
\in \Delta_{n-1}(q).
\]
This map can be extended to a homomorphism of Hecke algebras (cf. \cite{F}, \cite{K5}, \cite{K6}, \cite{K2}). The main result is 

\begin{corollary}\label{Corollary_7} 
If $p$ is an inert prime and $n\geq 2$ one has 
\[
 \phi_k(\Tcal_n(p)) = (p^{2n-1-k} +1) \Tcal_{n-1}(p).
\]
\end{corollary}
Note that we also need the Hecke algebra for $\Gamma_n[r]$, i.e.
\[
 \Tcal_n^r(p) = \Gamma_n[r] \begin{pmatrix}
                             I & 0 \\ 0 & pI
                            \end{pmatrix} \Gamma_n[r].
\]
If $p\equiv 1(\bmod{\,r})$ we have the same result as above due to \eqref{gl_1}.

\section{Hermitian modular forms}

Let 
\[
 \H_n:=\{Z\in\C^{n\times n};\;\tfrac{1}{2i}(Z-\overline{Z}^{tr})> 0\}
\]
denote the Hermitian half-space of degree $n$, where $>$ resp. $\geq0$ stands for positive definite resp. positive semi-definite. Given $f:\H_n\to \C$, $M=
\left(\begin{smallmatrix}
 A & B \\ C & D                                            
\end{smallmatrix}\right)\in\Delta_n$ we define for $k\in \Z$ 
\[
 f\underset{k}{\mid} M: \H_n\to \C, \; \; Z\mapsto \det(CZ+D)^{-k} f\bigl((AZ+B)(CZ+D)^{-1}\bigr).
\]
The vector space $\Mcal(\Gamma_n,k)$ of \emph{Hermitian modular forms} consists of all holomophic functions $f:\H_n\to \C$ satisfying
\[
 f\underset{k}{\mid} M = f\;\;\text{for all}\;\; M\in \Gamma_n
\]
with the usual additional condition of boundedness for $n=1$, where we deal with classical elliptic modular forms for $SL_2(\Z)$. Each $f\in \Mcal(\Gamma_n,k)$ possesses a Fourier expansion of the form
\[
 f(Z) = \sum_{T\in\Lambda_n,\,T\geq 0} \alpha_f(T) \,e^{2\pi i\trace (TZ)},
\]
where $T= (t_{ij}) \in \Lambda_n$ means $T=  \overline{T}^{tr}$,
\[
 t_{jj}\in\Z,\;\; t_{ij} \in \frac{1}{\sqrt{d_\K}} \oh_\K \;\; \text{for}\;\; i\neq j.
\]
The subspace of \emph{cusp forms} $\Ccal(\Gamma_n,k)$ is characterized by 
\[
 \alpha_f(T) \neq 0 \Rightarrow T>0.
\]
Moreover we define the \emph{Siegel $\phi$-operator} by
\[
 f\mid \phi:\H_{n-1} \to \C, \;\; Z_1 \mapsto \lim\limits_{y\to \infty} f \begin{pmatrix}
                                                                           Z_1 & 0 \\ 0 & iy
                                                                          \end{pmatrix}
  = \sum_{T_1\in \Lambda_{n-1},T_1\geq 0} \alpha_f \begin{pmatrix}
                                                  T_1 & 0 \\ 0 & 0
                                                  \end{pmatrix}
 e^{2\pi i\trace (T_1Z_1)}.
\]
If $h_\K = 1$ then $f$ is a cusp form if and only if $f\mid\phi \equiv 0$. This is more complicated for $h_\K> 1$ (cf. \cite{B4}, Lemma 1). Therefore let
\[
 R_U = \begin{pmatrix}
        \overline{U}^{tr} & 0 \\ 0 & U^{-1}
       \end{pmatrix} \in U(n,n;\K)\;\; \text{for} \; \; U\in GL_n(\K).
\]

\begin{theorem}\label{Theorem_3} 
Let $n\geq 2$ and let $\Ical_j = \langle u_j,1\rangle$, $u_j\in\K$, $j=1,\ldots, h$, $h=h_\K$, be a set of representatives of the ideal classes in $\K$. Then $f\in\Mcal(\Gamma_n,k)$ is a cusp form if and only if
\[
 f\underset{k}{\mid} R_{U_j}^{(n)} \mid \phi \equiv 0,\;\; U_j = 
 \begin{pmatrix}
  1 & 0 & \cdots & 0 \\ \vdots & \ddots & & \vdots \\ \vdots &  &\ddots & 0 \\ 0 & \cdots & 
  \overline{u}_{j} & 1
 \end{pmatrix}, \; j=1,\ldots,h.
\]
\end{theorem}

\begin{proof}
Let $T_0\in\Lambda_n$, $T_0\geq 0$, $\det T_0 = 0$. Then there exists $0\neq g\in \oh_\K^n$ with $T_0 g = 0$. Next we determine $U\in GL_n(\oh_\K)$ and $1\leq j \leq n$ such that 
\[
 \overline{U}^{tr-1} g= \begin{pmatrix}
                         0 \\ \vdots \\ 0 \\ u_j \\ 1
                        \end{pmatrix}
 \cdot \lambda, \;\; 0\neq \lambda \in \oh_\K,
\]
hence
\[
 g = \overline{U}^{tr} \overline{U}^{tr}_j \,e_n\cdot \lambda, \;\; T_0\bigl[\overline{U}^{tr}\overline{U}^{tr}_j\bigr] = 
 \begin{pmatrix}
  \ast & 0 \\ 0 & 0
 \end{pmatrix},  \;\; \ast \in\K^{(n-1)\times (n-1)}.
\]
In view of 
\begin{align*}
 & f\underset{k}{\mid} R_{U_j} = f\underset{k}{\mid} R_U \underset{k}{\mid}R_{U_j} = f \underset{k}{\mid}R_{U_j U} \\
 & = (\det U)^k  \sum_{T\in\Lambda,\,T\geq 0} \alpha_f(T) \,e^{2\pi i \trace(U_j UT\overline{U}^{tr} \overline{U}^{tr}_j \cdot Z)}
\end{align*}
the application of $\phi$ yields $\alpha_f(T_0)=0$. Hence $f$ is a cusp form.
\end{proof}

Now we have a closer look at the choice of $u_j$ in Theorem \ref{Theorem_3}.

\begin{lemma}\label{Lemma_6} 
Let $d_\K\neq -4,-8$ and $p$ be an odd prime, $p\mid d_\K$. Then representatives of the ideal classes $\Ical_j = \langle u_j,1\rangle$, $u_j\in \K$, may be chosen such that we find an $N\in\N$ with the properties
\[ 
p\nmid N \;\; \text{and}\;\; Nu_j\in\oh_\K,\; j=1,\ldots, h_\K.
\]
\end{lemma}

\begin{proof}
According to \cite{Fo}, p. 211, $u_j$ may be chosen in the form
\[
 u_j = \frac{\beta_j + \sqrt{d_\K}}{2\alpha_j}, \;\; \beta_j^2 -d_\K = 4\alpha_j \gamma_j, \;\; \alpha_j,\gamma_j \in \N,\;\beta_j\in\Z.
\]
As $\alpha_j\in\N$ let $N_j\in\N$ be minimal  such that $N_j u_j\in \oh_\K$, we may assume $p\mid \alpha_j$ as we are done otherwise. Then $p\mid \beta_j$ follows. As $p^2\nmid d_\K$ we obtain
\[
 p^2\nmid(\beta_j^2 - d_\K), \;\, p^2\nmid \alpha_j.
\]
Thus we may choose
\[
 u_j^* = \frac{2 \alpha_j}{\beta_j+\sqrt{d_\K}}, \;\; \langle u_j^*,1\rangle \K^* = \langle u_j,1\rangle \K^* 
\]
and 
\[
 N_j = \frac{\beta_j^2 - d_\K}{p}\in \N \;\;\text{satisfies}\;\; N_j u_j^* \in \oh_\K,\;\, p\nmid N.
\]
Then $N=N_1\cdot\ldots\cdot N_{h_\K}$ is a solution.
\end{proof}

Next we need a purely number theoretical assertion on the existence of such primes.

\begin{lemma}\label{Lemma_7} 
Let $d_\K \neq -4,-8$ and suppose that there is an odd prime divisor of $d_\K$, which does not divide $N\in \N$. Then there exist infinitely many inert primes $p\equiv 1\!\! \mod N$. 
\end{lemma}

\begin{proof}
At first assume $m\equiv 3 \,(\bmod 4)$. Let $\ell= \gcd (N,m)$. Then $m\neq \ell$ because of $m\nmid N$. We find $a\in \N$ with $\left(\frac{a}{m/\ell}\right) = -1$. Dirichlet's prime number theorem asserts the existence of infinitely many primes $p$ satisfying
\[
 p\equiv 1\,(\bmod \,4N), \;\; p\equiv a\,(\bmod \,m/\ell),
\]
since the modules are coprime. Quadratic reciprocity yields
\[
\chi_\K (p) = \left(\frac{-m}{p}\right) = \left(\frac{-1}{p}\right)\left(\frac{p}{m/\ell}\right) \left(\frac{p}{\ell}\right) = -1.
\]
The other cases are dealt with in a similar way.
\end{proof}

\section{Hecke operators}

Given $f\in \Mcal(\Gamma_n,k)$ we define the Hecke operator $\Gamma_n M \Gamma_n$, $M\in \Delta_n$, acting on $f$ by
\begin{gather*}\tag{3}\label{gl_3}
f\underset{k}{\mid} \Gamma_n M \Gamma_n = \sum_{L:\Gamma_n\backslash \Gamma_n M \Gamma_n} f\underset{k}{\mid} L \in \Mcal(\Gamma_n,k).
\end{gather*}
This definition is linearly extended on $\Hcal(\Gamma_n,\Delta_n)$. Moreover we apply the analogous definition for subgroups of $\Gamma_n$.

\begin{lemma}\label{Lemma_8} 
Hecke operators map cusp forms on cusp forms.
\end{lemma}

\begin{proof}
We may choose $L = \left(\begin{smallmatrix}
                          A & B \\ 0 & D
                         \end{smallmatrix}\right)$ 
in \eqref{gl_3} due to Lemma \ref{Lemma_2}.
\begin{gather*}\tag{4}\label{gl_4}
f\underset{k}{\mid} \left(\begin{smallmatrix}
                          A & B \\ 0 & D
                         \end{smallmatrix}\right)(Z) = \sum_{T\in\Lambda_n,\,T> 0}
 (\det D)^{-k} \alpha_f(T) \,e^{2\pi i \trace(TBD^{-1}+T[A]Z/q)}
\end{gather*}
if $M\in\Delta_n(q)$. Hence only positive definite matrices appear in the Fourier expansion.
\end{proof}

Next we consider the eigenvalues of Hecke operators.

\begin{lemma}\label{Lemma_9} 
Let $p$ be an inert prime and let $f\in\Mcal_k(\Gamma_n,k)$ with $\alpha_f(0) \neq 0$ as well as  $f\underset{k}{\mid} \Tcal_n(p)=\lambda f$ for some $\lambda\in\C$. Then 
\[
 \lambda = \prod^n_{j=1} (p^{2j-1-k} + 1).
\]
\end{lemma}

\begin{proof}
 Use Corollary \ref{Corollary_7} as well as
 \begin{align*}
  f\underset{k}{\mid}\Tcal_n(p)\mid \phi & = f\mid \phi\underset{k}{\mid} \phi_k(\Tcal_n(p) )\\
  & = (p^{2n-1-k}+1) f\mid \phi \underset{k}{\mid} \Tcal_{n-1}(p).
 \end{align*}
 as well as
 \[
  f\mid \phi^n = \alpha_f(0).
 \]
After $n$ steps the result follows.
\end{proof}

Next we consider the other extreme case of cusp forms.

\begin{lemma}\label{Lemma_10} 
Let $f\in \Mcal(\Gamma_n[q],k)$ be a cusp form. Let $p$ be an inert prime, $p\equiv 1(\bmod q)$, $M\in\Delta_n(p)$ and $f\underset{k}{\mid} \Gamma_n[q] M \Gamma_n[q] = \lambda f$. Then 
\[
 |\lambda| \leq p^{-kn/2} \prod^n_{j=1} (p^{2j-1}+1).
\]
\end{lemma}

\begin{proof}
There exists $Z_0\in \H_n$ such that the function
\[
\H_n\to \R, \;\; Z\mapsto (\det Y)^{k/2} |f(Z)|,
\]
attains its maximum at $Z_0$ due to \cite{B4}. Then the result follows in the same way as in \cite{F}, Hilfssatz IV.4.8, because of 
\[
 \sharp\bigl(\Gamma_n[q]\backslash \Gamma_n[q] M \Gamma_n[q]\bigr) = \sharp\bigl(\Gamma_n\backslash \Tcal_n(p)\bigr) = \prod^n_{j=1} (p^{2j-1} + 1)
\]
due to \eqref{gl_2} as well as the case $k=0$ in Lemma \ref{Lemma_9}.
\end{proof}

Next we need an assertion on iterative $\phi$-operators.

\begin{lemma}\label{Lemma_11} 
Let $f\in \Mcal(\Gamma_n,k)$, $R_j\in U(j,j;\K)$, $j=1,\ldots,n$. Then
\[
 f \underset{k}{\mid} R_n \mid \phi \underset{k}{\mid} R_{n-1}\mid \phi \ldots \;\underset{k}{\mid} R_1 \mid \phi =c \lim_{y\to \infty} f(iyI) = c\cdot \alpha_f(0)
\]
for some $c\neq 0$.
\end{lemma}

\begin{proof}
As $f\underset{k}{\mid} R_n \mid \phi \underset{k}{\mid} R_{n-1} = f \underset{k}{\mid} R_n(R_{n-1}\times I)\mid \phi$ (cf. \cite{K3}) we get
\[
 f \underset{k}{\mid} R_n \mid \phi \ldots \; \underset{k}{\mid} R_1 \mid \phi = f \underset{k}{\mid} R \mid \phi^n,
\]
where 
\[
 R = R_n \cdot (R_{n-1}\times I) \cdot \ldots \cdot (R_1\times I) \in U(n,n;\K).
\]
Now use Lemma \ref{Lemma_2} and \eqref{gl_4}.
\end{proof}

We give an application to the characterization of cusp forms. Therefore we use the special matrices $R_{U_\ell^{(n)}}$ from Theorem \ref{Theorem_3}.

\begin{lemma}\label{Lemma_12} 
Let $f \in\Mcal(\Gamma_n,k)$, $1\leq j\leq n$. Then 
\[
f \underset{k}{\mid} R \mid \phi^j \equiv 0 \;\; \text{for all}\;\; R\in U(n,n;\K)
\]
holds if and only if this is true for
\begin{gather*}\tag{5}\label{gl_5}
 R =  R_{U_{i_n}^{(n)}} R_{V_n} \cdot \left( R_{U_{i_{n-1}}^{(n-1)}} R_{V_{n-1}}\times I\right) \cdot \ldots \cdot \left(R_{U_{i_{n-j+1}}^{(n-j+1)}} R_{V_{n-j+1}} \times I\right)\in U(n,n;\K)
\end{gather*}
for all $V_\ell \in GL_\ell (\oh_\K)$ and $i_\ell \in \{1,\ldots,h_\K\},$ $\ell = n,\ldots,n-j+1$.
\end{lemma}

\begin{proof}
Apply the same arguments as in the proof of Theorem \ref{Theorem_3} and Lemma \ref{Lemma_11}.
\end{proof}

\begin{remark} 
 If $f\in\Mcal(\Gamma_n,k)$ is symmetric, i.e. $f(Z^{tr}) = f(Z)$, and $M\in\Delta_n$ with $\det M\in\R_+$, we observe
 \[
  f\underset{k}{\mid} \Gamma_n M\Gamma_n (Z^ {tr}) = f\underset{k}{\mid} \Gamma_n \overline{M} \Gamma_n (Z).
 \]
We conclude $\Gamma_n \overline{M} \Gamma_n = \Gamma_n M \Gamma_n$ for $M\in\Delta_n^{\text{inert}}$ from Theorem \ref{Theorem_2}. Thus these Hecke operators map the subspace of symmetric Hermitian modular forms on itself.
\end{remark}

\section{The Siegel-Eisenstein series}

According to \cite{B} we may define the Siegel-Eisenstein series
\[
E_k^{(n)} (Z) = \sum_{M:\Gamma_{n,0} \backslash \Gamma_n} 1 \underset{k}{\mid} M(Z),\;\, Z\in\H_n,
\]
for even $k>2n$, $d_\K\neq -3,-4$, where
\[
\Gamma_{n,0} = \left\{\begin{pmatrix}
											A & B \\ 0 & D
											\end{pmatrix}
\in \Gamma_n\right\}.
\]
We have 
\[
E_k^{(n)} \mid \phi = E_k^{(n-1)}, \;\; E_k^{(0)} : = 1.
\]
We can take the same proof as in \cite{F}, IV.4.7, in order to get

\begin{lemma}\label{Lemma_13} 
Let $k>2n$ be even, $d_\K\neq -3,-4$, $M\in\Delta_n$. Then there exists a $\lambda\in\C$ such that
\[
E_k^{(n)} \underset{k}{\mid} \Gamma_n M \Gamma_n = \lambda E_k^{(n)}.
\]
\end{lemma}
We obtain our final result and recall the definition of $N$ from Lemma \ref{Lemma_6}.

\begin{theorem}\label{Theorem_4} 
Let $k>2n$, $d_\K \neq -3,-4$. Let $p$ be an inert prime 
\[
p\equiv 1 \!\!\! \mod N^{2n-2}
\]
and $f\in \Mcal_k(\Gamma_n,k)$ satisfying
\[
\alpha_f(0) = 1 \;\; \text{and}\;\; f \underset{k}{\mid} \Tcal_n(p) = \lambda f
\]
for some $\lambda \in\C$. Then
\[
f= E_k^{(n)}.
\]
\end{theorem}

\begin{proof}
The case $n=1$ is clear from the classical theory as $E_k^{(1)}$ coincides with the normalized elliptic Eisenstein series. Let $n\geq 2$. Since the constant term of the Fourier expansion is non-zero, we can apply Lemma \ref{Lemma_9}. If $f\neq E_k^{(n)}$, there exists a minimal $j$, $1\leq j\leq n$ such that
\[
(f-E_k^{(n)}) \underset{k}{\mid} R \mid \phi^j \equiv 0 \;\; \text{for all}\;\; R\in U(n,n;\K).
\]
This means that the non-zero Fourier coefficients have rank $>n-j$. Now apply Lemma \ref{Lemma_12} and assume
\[
\widetilde{f} : = (f-E_k^{(n)}) \underset{k}{\mid} R\mid \phi^{j-1} \not\equiv 0.
\]
for an $R\in U(n,n;\K)$ of the form \eqref{gl_5} quoted there. Thus $\widetilde{f} \in \Mcal(\Gamma_{n-j+1} [N^{2j-2}],k)$ is a cusp form. We conclude from Lemma \ref{Lemma_9}
\begin{align*}
& \widetilde{f} \underset{k}{\mid} \Tcal_{n-j+1}^{N^{2j-2}} (p) = \lambda \widetilde{f}, \\
& \lambda = \prod\limits_{\ell=1}^{n-j-1} (p^{2\ell -1-k} +1) > 1.
\end{align*}
But $\widetilde{f}$ is a cusp form. Therefore we can apply Lemma \ref{Lemma_10} in order to get 
\begin{align*}
|\lambda| & \leq p^{-k(n-j+1)/2} \prod\limits_{\ell=1}^{n-j+1} (p^{2\ell-1} +1) < p^{-k(n-j+1)/2} \prod\limits_{\ell=1}^{n-j+1} p^{2\ell}  \\
& = p^{(n-j+1)(n-j+2-k/2)} \leq 1
\end{align*}
in view of $k>2n$. This contradicts $\lambda>1$ and yields the claim.
\end{proof}

\begin{remark} 
a) The cases $d_\K = -3,-4$ are excluded because of the additional units. As $h_\K=1$ in these cases, the results are contained in \cite{K4}, where the proof is only valid for class number $1$. Due to our proof here the results in \cite{NN} are also valid for arbitrary $\K$. Moreover these considerations fill the gap in \cite{K4} such that the results of section 8 there are true for arbitrary $h_\K$. 
\\[1ex]
b) If $d_\K=-3,-4$ one has to impose the condition that $k$ is divisible by the number of units in $\oh_\K$. Alternatively for arbitrary even $k$ one has to restrict the summation to $\Gamma_n\cap SL_{2n}(\oh_\K)$ or to insert the factor $(\det M)^{-k/2}$ in the definition of $E_k^{(n)}$.
\end{remark}


\bibliographystyle{plain}
\renewcommand{\refname}{Bibliography}
\bibliography{bibliography_krieg} 

\begin{thebibliography}{10}

\bibitem{B}
{Braun, H.}
\newblock {Hermitian modular functions.}
\newblock {\em Ann. Math.}, 50:827--855, 1949.

\bibitem{B2}
{Braun, H.}
\newblock {Hermitian modular functions III.}
\newblock {\em Ann. Math.}, 53:143--160, 1950.

\bibitem{B4}
{Braun, H.}
\newblock {Darstellung hermitischer Modulformen durch Poincarésche Reihen.}
\newblock {\em Abh. Math. Semin. Univ. Hamb.}, 22:9--37, 1958.

\bibitem{De}
{Dern, T.}
\newblock {Multiplikatorsysteme und Charaktere Hermitescher Modulgruppen}.
\newblock {\em Monatsh. Math.}, 126:109--116, 1998.

\bibitem{En}
{Ensenbach, M.}
\newblock {\emph{Hecke-Algebren zu unimodularen und unitären Matrixgruppen.}}
\newblock
  {http://publications.rwth-aachen.de/record/50390/files/Ensenbach\_Marc.pdf}.
\newblock {PhD thesis, RWTH Aachen, 2008}.

\bibitem{En2}
{Ensenbach, M.}
\newblock {Determinantal divisors of products of matrices over Dedekind
  domains.}
\newblock {\em Linear Algebra Appl.}, 432:2739--2744, 2010.

\bibitem{Fo}
{Forster, O.}
\newblock {\em {Algorithmische Zahlentheorie. 2. Aufl.}}
\newblock Springer-Verlag, Berlin, 2015.

\bibitem{F}
{Freitag, E.}
\newblock {\em {Siegelsche Modulfunktionen}}, volume 254 of {\em {Grundl. Math.
  Wiss.}}
\newblock Springer-Verlag, Berlin, 1983.

\bibitem{K3}
{Krieg, A.}
\newblock {\em {Modular forms on half-spaces of quaternions}}, volume 1143 of
  {\em Lect. Notes Math}.
\newblock \!\!\!, Springer-Verlag, Berlin, 1985.

\bibitem{K5}
{Krieg, A.}
\newblock {Das Vertauschungsgesetz zwischen Hecke-Operatoren und dem
  Siegel\-schen $\phi$-Operator.}
\newblock {\em Arch. Math.}, 46:323--329, 1986.

\bibitem{K6}
{Krieg, A.}
\newblock {The Hecke-algebras related to the unimodular and modular group over
  the Hurwitz order of integral quaternions.}
\newblock {\em Proc. Indian Acad. Sciences - Math. Sciences.}, 97:201--229,
  1987.

\bibitem{K2}
{Krieg, A.}
\newblock {Hecke algebras}.
\newblock {\em Mem. Am. Math. Soc.}, 435, 1990.

\bibitem{K4}
{Krieg, A.}
\newblock {The Maaß spaces on the Hermitian half-space of degree $2$}.
\newblock {\em Math. Ann.}, 289:663--681, 1991.

\bibitem{Ma1}
{Manickam, M.}
\newblock {On Hecke theorie for Jacobi forms}.
\newblock In T.N. Venkataramana, editor, {\em Cohomology of arithmetic groups,
  $L$-functions and automorphic forms}, pages 89--93. {Narosa Publishing House,
  New Delhi} edition, 2001.

\bibitem{NN}
{Nagaoka, S. and Y. Nakamura}.
\newblock {On the restriction of the Hermitian Eisenstein series and its
  applications.}
\newblock {\em Proc. Am. Math. Soc.}, 139:1291--1298, 2011.

\bibitem{Rm}
{Raum, M.}
\newblock {Hecke algebras related to the unimodular and modular groups over
  quadratic field extensions and quaternion algebras.}
\newblock {\em Proc. Am. Math. Soc.}, 139:1301--1331, 2011.

\end{thebibliography}

\end{document}